\documentclass{article}
\pdfoutput=1
\usepackage{graphicx}
\usepackage{amsthm}
\usepackage{enumerate}
\usepackage[usenames, dvipsnames]{color}
\usepackage{hyperref}

\theoremstyle{definition}
\newtheorem{definition}{Definition}[section]
\newtheorem{example}[definition]{Example}

{
\theoremstyle{plain}
\newtheorem{proposition}[definition]{Proposition}
\newtheorem{lemma}[definition]{Lemma}
\newtheorem{theorem}[definition]{Theorem}
}

\author{\textbf{Marek Hy\v{c}ko}\footnote{{\bf
Acknowledgement} The support of the grants VEGA 2/0059/12 and APVV-0178-11 is kindly announced.
}\\Mathematical Institute of Slovak Academy of Sciences\\\v{S}tef\'{a}nikova 49, SK-81473 Bratislava, Slovakia}
\title{Weak pre pseudo effect algebras and generalized weak pre pseudo effect algebras}
\date{}

\begin{document}
\maketitle
\begin{abstract}
Present article concerns the generalization of %some of 
the results of I. Chajda and J. K\"{u}hr to a non-commutative setting resulting into so called (weak) pre pseudo effect algebras and generalized pre pseudo effect algebras. %Moreover, computer methods for generating finite models are discussed.
\end{abstract}

\section{Introduction}
The article concerns a generalization of pseudo effect algebras and generalized pseudo effect algebras. Pseudo effect algebras were introduced and studied by T. Vetterlein and A. Dvure\v{c}enskij in series of papers \cite{DvVe} as a non-commutative generalization of effect algebras introduced by D. Foulis and M. K. Bennet \cite{FoBe} or equivalently of D-posets introduces by F. K\^{o}pka and F. Chovanec \cite{KoCh}.

Main motivation was to addapt %generalize 
results of I. Chajda and J. K\"{u}hr \cite{ChaKu} to a non-commutative setting. The authors introduced the generalization of effect algebras. Their motivation was to find a structure for which the underlying lattice orders define an ortholattice. It is well known that when lattice ordered orthoalgebra is assumed then underlying structure (order, orthocomplementation as orthosupplement, constants) is an orthomodular lattice. Thus in order to obtain ortholattices some tweaking needed to be performed. It turnes out that it is sufficient to omit the condition that orthosupplement is the only element which sums with corresponding element to a unit element. The newly defined structure was called pre-effect algebra and pre-orthoalgebra.

%The results in this paper generalize the approach of I. Chajda and J. K\"{u}hr to a non-commutative setting. 
Commutative version of a presented generalization leads to the structure defined by I. Chajda and J. K\"{u}hr.

Moreover, in \cite{ChaKu} the structures without top element were studied, generalizing generalized effect algebras. Also these results were adapted to a non-commutative setting in this contribution. %This contribution contains also a generalization to a non-commutative setting. 
It is shown that the well known construction of unitization \cite{HePu}, \cite{PuVi} can also be performed for these structures.

%Last part of the article contains the description of methods that lead to generate all finite models of weak pre pseudo effect algebras up to 11 elements. The approach is unique in such a way that first methods were not able in reasonable time to produce all models on 7 elements. After some modifications, the improved methods were developed and it is now possible to obtain in reasonable time (less that week of computation) all models up to 11 elements.

The article is organized as follows, Section 2 contains definition of pre pseudo effect algebras with some basic results, Section 3 concerns weak generalized pre pseudo effect algebras and finally Section 4 contains attempts to define congruences on such structures.

\section{Preliminaries}

First recall definitions of a pre-effect algebra, \cite[Definition 2.1]{ChaKu} and a generalized pre-effect algebra, \cite[Definition 3.1]{ChaKu}.

\begin{definition}
A \emph{pre-effect algebra} is a structure $(A; +, {}', 0, 1)$ where $(A; +, 0)$
is a partial abelian monoid, 1 is an element of A and is a unary operation such
that $a + a' = 1$ for all a ∈ A, and the relation $\leq$ given by the rule

$$a \leq b\mbox{ iff } a + b'\mbox{ is defined}$$
is a partial order. A pre-effect algebra satisfying the condition that $a = 0$
whenever a + a is defined (i.e. $a = 0$ if $a \leq a'$ ) is called a \emph{pre-orthoalgebra}.
\end{definition}

\begin{definition}
A \emph{generalized pre-effect algebra} is a structure $(A; +, -, 0)$
where $+$ and $-$ are partial binary operations on A such that
\begin{enumerate}[(GQEX)]\itemsep=-1mm
\item[(GQE1)] $+$ is commutative, i.e., $a + b = b + a$ if one side is defined,
\item[(GQE2)] $a - a = 0$ for all a ∈ A,
\item[(GQE3)] the relation $\leq$ defined by $a \leq b$ iff $b - a$ exists is a partial order,
\item[(GQE4)] for all $a, b, c \in A$, $a \geq b$ and $a - b \geq c$ iff $b + c$ is defined and $a \geq b + c$,
in which case $(a - b) - c = a - (b - c)$.
\end{enumerate}
\end{definition}

\section{(Weak) Pre pseudo effect algebras}

First we give a definition of a weak pre pseudo effect algebra.

\begin{definition}\label{def:wprepea1}
Let $(A; \oplus, {}^L, {}^R, 0, 1)$ be a partial algebra of type $(2, 1, 1, 0, 0)$ satisfying the following properties:
\begin{enumerate}[({WPPEA}1)]\itemsep=-1mm
\item $\oplus$ is partially associative, i.e. for any $a, b, c\in A$: $a\oplus b$ and $(a\oplus b)\oplus c$ are defined, iff $b\oplus c$ and $a\oplus (b\oplus c)$ are defined and in such case $(a\oplus b)\oplus c = a\oplus (b\oplus c)$;
\item $a\oplus a^R = 1 = a^L\oplus a$;
\item relation $a\leq b$, iff $a\oplus b^R$ is defined, iff $b^L\oplus a$ is defined is a partial order;
\item $1\oplus a$, or $a\oplus 1$ is defined, then $a = 0$;
\item $0$ and $1$ are comparable to all elements of $A$.
\end{enumerate}
The $A$ is called a \emph{weak pre pseudo effect algebra}. If moreover $A$ satisfies the condition:
\begin{itemize}
\item[(PEA)] if $a\oplus b$ is defined then there are $e, f\in A$ such that $a\oplus b = e\oplus a = b\oplus f$,
\end{itemize}
Then $A$ is called \emph{pre pseudo effect algebra}.
\end{definition}

Let us consider a relation $\sqsubseteq_R$ defined in the following way: $a \sqsubseteq_R b$, iff there is $c\in A$ such that $a = b \oplus c$ (summing from the right). Similarly, we defined $\sqsubseteq_L$ as $a\sqsubseteq_L b$, iff there is $d\in A$ such that $a = d \oplus b$ (summing from the left). It is not hard to show that these relations are partial orders. Unlike in the case of effect algebras, these partial orders need not be the same and also need not coincide with $\leq$ order. Condition (PEA) ensures that $\sqsubseteq_L = \sqsubseteq_R$.  Let us recall that (PEA) conditions was originally defined by A. Dvure\v{c}enskij and T. Vetterlein in the setting of pseudo effect algebras.

If we consider partial commutativity condition, i.e. $a \oplus b$ is defined, iff $b \oplus a$ is defined and $a \oplus b = b\oplus a$, then it implies (PEA) property trivially. Thus each commutative weak pre pseudo effect algebra is a commutative pre pseudo effect algebra, i.e. pre effect algebra in the sense of I. Chajda and J. K\"uhr.

The following properties are base for the alternative definition of WPPEAs.

\begin{lemma}\label{lem:wprepea-defs}
Let $A$ be a weak pre pseudo effect algebra according to Definition \ref{def:wprepea1}. Then the following properties hold.
\begin{enumerate}[(i)]\itemsep-1mm
\renewcommand{\labelenumi}{{\rm (\theenumi)}}
\item $1^L = 0 = 1^R$;
\item if $a^R = b^R$, then $a = b$; if $a^L = b^L$, then $a = b$;
\item $1$ is top element in $(A;\leq)$, i.e. for any $b\in A$, $b\leq 1$;
\item $a\oplus 0$ and $0\oplus a$ are defined. Moreover, $a\oplus 0 = a = 0\oplus a$;
\item $0^L = 1 = 0^R$;
\item $0$ is bottom element in $(A; \leq)$, i.e. for any $a\in A$, $0\leq a$.
\end{enumerate}
\end{lemma}

\begin{proof}
(i) follows from (WPPEA2) and (WPPEA4), since $1 \oplus 1^R$ is defined, then $1^R = 0$.

Let us consider $a^R = b^R$. Then $a \oplus a^R$ and $b\oplus b^R$ are defined, thus $a \oplus b^R$ and $b\oplus a^R$ are defined. From (WPPEA3) follows $a\leq b$ and $b\leq a$, i.e. $a = b$. 

From (WPPEA5) $1$ is comparable to all elements and let us consider that $1\leq c$ for some $c\in A$. Then $1 \oplus c^R$ is defined, thus $c^R = 0 = 1^R$. From (ii) follows that $c = 1$ and thus $1$ is the top element.

From (iii) $a\leq 1$ for any $a\in A$. Thus $a\oplus 1^R = a\oplus 0$ and $1^L\oplus a = 0\oplus a$ are defined. From (WPPEA1) since $(a\oplus 0)^L\oplus (a\oplus 0) = 1$ is defined then also $(a\oplus 0)^L\oplus a$ is defined. So we have that $a\leq (a\oplus 0)$. On the other hand $1 \oplus 0 = (a^L\oplus a)\oplus 0$ is defined and again from (WPPEA1) we get $a^L\oplus (a\oplus 0)$ is defined, thus $a\oplus 0\leq a$. This proves $a = a\oplus 0$.

From (WPPEA2) and (iv) follows $1 = 0 \oplus 0^R = 0^R$.

$0$ is according (WPPEA5) comparable to any element of $A$. Let $a\leq 0$, then $a\oplus 0^R = a\oplus 1$ is defined, i.e. from (WPPEA4) $a = 0$.
\end{proof}

From the previous lemma follows that conditions (WPPEA1) and (WPPEA5) can be replaced by the condition:
\begin{enumerate}[({WPPEA}1{'})]
\item $(A; \oplus, 0)$ is a partial monoid with neutral element $0$.
\end{enumerate} 

For the converse, let us consider that conditions (WPPEA1'), (WPPEA2)-(WPPEA4) hold. Since $0$ is neutral element of a partial monoid $A$, $0\oplus a^R$ and $a^L\oplus 0$ are defined for any element $a\in A$, thus $0$ is comparable (less than or equal) to any element in $A$. Condition (i) of Lemma \ref{lem:wprepea-defs} is proved by the use of properties (WPPEA2) and (WPPEA4) and thus $a\oplus 1^R = a\oplus 0$ and $1^L\oplus a = 0\oplus a$ are defined for any $a\in A$, thus also $1$ is comparable to any element of $A$. This proves (WPPEA5). (WPPEA1) is trivially hidden in (WPPEA1') since any partial monoid is partially associative.

Thus weak pre pseudo effect algebras and pre pseudo effect algebras can be alternatively defined as follows.

\begin{definition}
Let $(A; \oplus, {}^L, {}^R, 0, 1)$ be a partial algebra of type $(2, 1, 1, 0, 0)$ satisfying properties (WPPEA1') and (WPPEA2)-(WPPEA4).
%\begin{itemize}\itemsep-1mm
%\item $(A; \oplus, 0)$ is partial monoid with neutral element $0$;
%\end{itemize} and (WPPEA2)-(WPPEA4).
%\item $a\oplus a^R = 1 = a^L\oplus a$;
%\item relation $a\leq b$, iff $a\oplus b^R$ is defined, iff $b^L\oplus a$ is defined is a partial order;
%\item $1\oplus a$, or $a\oplus 1$ is defined, then $a = 0$;
%\end{itemize}
Then $A$ is called a \emph{weak pre pseudo effect algebra}. If moreover $A$ satisfies the condition (PEA), then
%\begin{itemize}
%\item if $a\oplus b$ is defined then there are $e, f\in A$ such that $a\oplus b = e\oplus a = b\oplus f$,
%\end{itemize}
%Then
$A$ is called \emph{pre pseudo effect algebra}.
\end{definition}

%To prove such connection, the following lemma supplies the necessary steps:

For multiple applications of unary orthosupplement operations we will use the following abbreviation: e.g. for $(a^R)^L$ we will shortly write $a^{RL}$ instead. $a^{LR}$ and other cases are defined similarly.

Properties of (weak) pre pseudo effect algebras can be summarized in the following proposition.
\begin{proposition}
Let $A$ be a weak pre pseudo effect algebra. Then
\begin{enumerate}[(i)]\itemsep-1mm
\renewcommand{\labelenumi}{{\rm (\theenumi)}}
\item $a^{LRL} = a^L$, $a^{RLR} = a^R$;

\item $a^{RL} = a = a^{LR}$%, where $a^{RL}$ is the abbreviation of $\left(a^R\right)^L$ and $a^{LR}$ is defined similarly
;
\item if $a\oplus b$ is defined, then $a, b\leq a\oplus b$;
\item if $a\oplus b$ is defined and $a\oplus b = a$, then $b = 0$;
\item if $a\oplus b$ is defined and $a\oplus b = b$, then $a = 0$;
\item if $a\oplus b$ is defined and $a\oplus b = 0$, then $a = b = 0$ (positivity of $\oplus$);
\item $a\leq b$, iff $b^R\leq a^R$, iff $b^L\leq a^L$;
\item if $b\leq c$ and $a\oplus c$ is defined, then $a\oplus b$ is defined and $a\oplus b\leq a\oplus c$;
\item if $b\leq c$ and $c\oplus a$ is defined, then $b\oplus a$ is defined and $b\oplus a\leq c\oplus a$;
\item each commutative weak pre pseudo effect algebra is pre pseudo effect algebra, moreover it is also pre effect algebra;
\item each cancelative (if $a \oplus b = a\oplus c$, then $b = c$ and if $b\oplus a = c\oplus a$, then $b = c$) pre pseudo effect algebra is pseudo effect algebra.
\end{enumerate}
\end{proposition}

\begin{proof}
Since $a^L\oplus a^{LR}$, $a^{RL}\oplus a^R$ are defined, we get that $a^{LR}, a^{RL}\leq a$. Substituing $a$ in the inequalities by $a^L$, $a^R$, $a^{LR}$, $a^{RL}$ we get $a^{LRL}\leq a^L$, $a^{RLR}\leq a^R$, $a^{LRLR}\leq a^{LR}$ and $a^{RLRL}\leq a^{RL}$. Thus $a^{RLRL}\leq a$, i.e. $a^{RLRL}\oplus a^R$ is defined, which also means that $a^R\leq a^{RLR}$, i.e. $a^{RLR} = a^R$. Similarly, it is shown that $a^{LRL} = a^L$.

Since $a\oplus a^R$ is defined, by (i) we have that $a\oplus a^{RLR}$ is defined and $a\leq a^{RL}$.  This concludes that $a = a^{RL}$. 

Since $(a\oplus b)^L\oplus (a\oplus b)$ and $(a\oplus b)\oplus (a\oplus b)^R$ are defined, from partial associativity we get that $(a\oplus b)^L\oplus a$ and $b\oplus (a\oplus b)^R$ are defined, which leads to desired inequalities.

Let $a\oplus b = b$, then $b\oplus b^R = (a\oplus b)\oplus b^R = a\oplus (b\oplus b^R) = a\oplus 1$ is defined, i.e. $a = 0$.

Similarly, if $a\oplus b = a$, then $a^L\oplus a = a^L\oplus (a\oplus b) = 1\oplus b$ is defined, i.e. $b = 0$.

Let $a\oplus b = 0$, then from (iii) $a, b\leq a\oplus b = 0$. Since $0$ is bottom element we have $a = b = 0$.

Let $a\leq b$, then $a\oplus b^R = a^{RL}\oplus b^R$ is defined, which means $b^R\leq a^R$. Conversely, let $a^R\leq b^R$, then $b^{RL}\oplus a^R = b\oplus a^R$ is defined, i.e. $b\leq a$. 

Let $b\leq c$ and $a\oplus c$ is defined. Then $c^L\leq b^L$ and since $a\oplus c^{LR}$ is defined, $a\leq c^L\leq b^L$ thus $a\oplus (b^L)^R = a\oplus b$ is defined. $(a\oplus c)^L\oplus (a\oplus c)$, thus $b\leq c\leq ((a\oplus c)^L\oplus a)^R$, i.e. $((a\oplus c)^L\oplus a)^{RL}\oplus b = (a\oplus c)^L\oplus (a\oplus b)$ is defined, so $a\oplus b\leq a\oplus c$.

(ix) is proved similarly as (viii).

Trivially, commutativity implies that (PEA) property is satisfied. Moreover, commutative pre pseudo effect algebra is pre effect algebra. 

From cancellativity it holds that since $a\oplus a^R = 1$, $a^R$ is the unique element such that $a\oplus a^R = 1$. Similarly, $a^L$ is the unique element that $a^L\oplus a = 1$. The (PEA) holds from assumptions. Thus $(A;\oplus, 0, 1)$ is pseudo effect algebra.

\end{proof}

\begin{example}
For $n = 6$, there are 16 non-isomorphic bounded partial orders with bottom and top elements:

\includegraphics{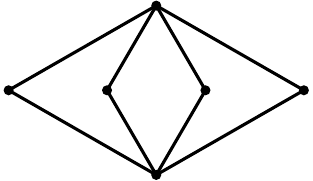}
\includegraphics{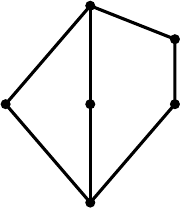}
\includegraphics{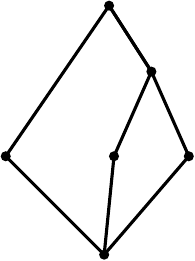}
\includegraphics{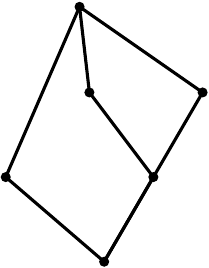}
\includegraphics{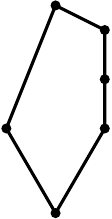}
\includegraphics{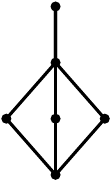}
\includegraphics{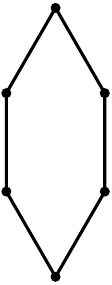}
\includegraphics{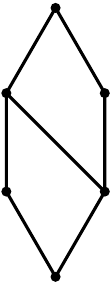}
\includegraphics{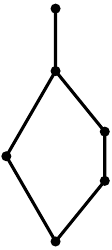}
\includegraphics{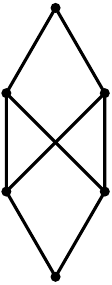}
\includegraphics{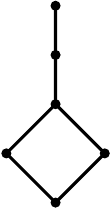}
\includegraphics{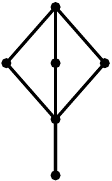}
\includegraphics{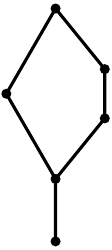}
\includegraphics{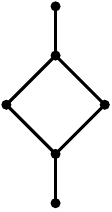}
\includegraphics{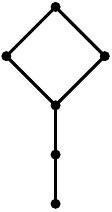}
\includegraphics{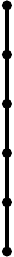}

The structure of a weak pre pseudo effect algebra can be defined only for the following orders:

\includegraphics{orders6_2_01.pdf}
\includegraphics{orders6_2_02.pdf}
\includegraphics{orders6_2_05.pdf}
\includegraphics{orders6_2_07.pdf}
\includegraphics{orders6_2_08.pdf}
\includegraphics{orders6_2_10.pdf}
\includegraphics{orders6_2_14.pdf}
\includegraphics{orders6_2_16.pdf}

%It is not hard to see that in order to defined orthosupplements, unlabeled Hasse diagrams of $(A,\leq)$ and $(A,\leq')$, where $a\leq' b$, iff $b\leq a$, can be drawn in the same way.

\end{example}

\begin{example} Let us consider a bounded poset with double orthocomplementation, i.e. the structure $(P; \leq, {}^L, {}^R, 0, 1)$, where $(P; \leq, 0, 1)$ is a bounded poset and for any $x, y\in P$ hold
\begin{itemize}
\item $x\leq y$, iff $y^L\leq x^L$, iff $y^R\leq x^R$;
\item $x^{LR} = x = x^{RL}$;
\item $0^L = 1 = 0^R$ and $1^L = 0 = 1^R$.
\end{itemize}
Let us equip this structure with a partial operation $\oplus$, which is defined as follows:
$x\oplus y$ is defined, iff $x\leq y^L$, iff $y\leq x^R$ and in such case $x\oplus y = 1$, if both $x, y\neq 0$; $x \oplus 0 = x = 0 \oplus x$, otherwise. Then $(P; \oplus, {}^L, {}^R, 0, 1)$ is a weak pre pseudo effect algebra.

If $1 \oplus a$ is defined, iff $1\leq a^L$, iff $a\leq 0$, i.e. for $a = 0$. Similarly, $a \oplus 1$ is defined, iff $a = 0$.

If $(a \oplus b)\oplus c$ is defined then at least one element of $a$, $b$, $c$ must be equal to $0$. In such case it is very easy to check associativity of $\oplus$.

$a \oplus a^R$ is defined, iff $a \leq a^{RL} = a$ and $a^L\oplus a$ is defined iff $a^L\leq a^L$. Since for any $a\neq 0,1$ are both $a$ and $a^R$ ($a^L$) different to $0$, the result of the sum is desired $1$. If $a = 0$, then $a^R$ ($a^L$) is $1$ and again the sum is $1$. If $a = 1$, then $a^R = 0$ ($a^L = 0$) and the sum is $1$.

The relation  $\leq'$ defined $a\leq' b$, iff $a\oplus b^R$ is defined, iff  $b^L\oplus a$ is defined coincides with $\leq$ relation of a poset $P$. Let $x\leq' y$, iff $x\oplus y^R$ is defined, iff $y^L\oplus x$ is defined. That is $x\leq y^{RL} = y$, ($x\leq y^{LR} = y)$. Let $x\leq y = y^{RL}$, then $x \oplus y^R$ is defined and $x\leq y^{LR}$, then $y^L\oplus x$ is defined, thus $x\leq' y$.

\medskip
Conversely, the ``poset''-reduct ($(\leq; {}^L, {}^R, 0, 1)$-) of each weak pre pseudo effect algebra is a bounded poset with double orthocomplementation.
\end{example}

\section{Generalized pre pseudo effect algebras}

\newcommand{\mR}{\mathop{\backslash}}
\newcommand{\mL}{\mathop{/}}

The rough interpretation of $a \mL b $ (left minus) is $-b + a$ in the group sense and $a \mR b$ (right minus) is $a + (-b)$ in the group sense.

\begin{definition}
Let $(A; +, \mR, \mL, 0)$ be a partial algebra of type $(2, 2, 2, 0)$ satisfying the following properties:
\begin{enumerate}[({GPPEA}1)]\itemsep-1mm
\item $a \mR a = 0 =  a \mL a$;
\item the relation $a\leq b$, iff $b \mR a$ is defined, iff $b \mL a$ is defined is a partial order;
\item $a\mR b$ is defined and $a \mR b \geq c$, iff $c + b$ is defined and  $a\geq c + b$. Moreover  $(a \mR b)\mR c = a\mR (c + b)$;
\item $a\mL b$ is defined and $a \mL b \geq c$, iff $b + c$ is defined and  $a \geq b + c$. Moreover $(a \mL b) \mL c = a \mL (b + c)$.
\end{enumerate}
Then $A$ is said to be a \emph{generalized pre pseudo effect algebra}.
\end{definition}

Each generalized pseudo effect algebra is generalized pre pseudo effect algebra.
\begin{proposition}\label{prop:gprepea}
Let $(A; +, \mR, \mL, 0)$ be a generalized pre pseudo effect algebra. Then
\begin{enumerate}[(i)]\itemsep-1mm
\renewcommand{\labelenumi}{{\rm (\theenumi)}}%
\item $a + 0$, $0 + a$ are defined and $a + 0 = a = 0 + a$;
\item $0$ is the bottom element in $(A; \leq)$;
\item $a\mR 0$ and $a\mL 0$ are defined and $a\mR 0 = a = a\mL 0$;
\item if $a + b$ is defined, then $a +b \geq a,b$ and $(a + b)\mR b \geq a$, $(a+ b)\mL a\geq b$;
\item if $a\mR b$ is defined, then $a\mR b\leq a$, $a\mL(a\mR b) \geq b$ and $a\geq (a\mR b) + b$;
\item if $a\mL b$ is defined, then $a\mL b\leq a$, $a\mR(a\mL b) \geq  b$ and $a\geq b + (a\mL b)$;
\item if $a + b = a$ (or $a + b = b$) then $b = 0$ ($a = 0$);
\item if $a + b = 0$, then $a = b = 0$;
\item $a\mL b = a$, iff $b = 0$; $a\mR b = a$, iff $b = 0$;
\item $+$ is partially associative, i.e. $a + b$ and $(a + b) + c$ are defined iff $b+c$ and $a + (b + c)$ are defined. In such case $(a + b) + c = a + (b + c)$;
\item if $a\geq b\geq c$ then $a\mR c\geq b\mR c$ and $a\mL c \geq b\mL c$;
\item if $a\geq b$ and $a + c$ is defined then $b+c$ is defined, $a+c\geq b+c$ and $(a + c)\mR (b+c)\geq a\mR b$;

\item if $a\geq b$ and $c + a$ is defined then $c+b$ is defined, $c+a\geq c+b$ and $(c + a)\mL (c+b)\geq a\mL b$;

\item $a\geq b\geq c$ then $a\mR c\geq (a\mR b) + (b\mR c)$ and $a\mL c\geq (b\mL c) + (a\mL b)$;

\item $(b\mR a)\mL c$ is defined, iff $(b\mL c)\mR a$ is defined and in this case $(b\mR a)\mL c = (b\mL c)\mR a$.
\end{enumerate}
\end{proposition}

\begin{proof}

Since $0 = a \mR a \geq 0$, iff $0+a$ is defined and $a\geq 0+a$ and $0 \mR 0 = (a\mR a)\mR 0 = a\mR (0+a)$. Similar reasoning gets us that $a+0$ is defined and $a\geq a+0$ and $0\mL 0 = (a\mL a)\mL 0 = a\mL(a+0)$. Since $a+0\geq a+0$ then $(a+0)\mL a$ is defined and it is greater than or equal to $0$. Thus $a+0\geq a$ and analogously $0+a \geq a$. Together with previous result we get that $0+a = a = a+0$. From the fact that $a\geq 0+a$ we have that $a\mL 0$ is defined, i.e. $0\leq a$ for any $a\in A$, thus $0$ is the bottom element in partial ordered set $(A; \leq)$. From $a \mR 0\geq a\mR 0$ we get $a\geq (a\mR 0) + 0$ and $a\mL(a\mR0)\geq 0$, i.e.$ a\geq (a\mR0)$. Thus $a = a\mR 0$. In a similar way we are able to prove $a = a\mL 0$.

From $a+b\geq a+b$ we get that $(a+b)\mR b$ and $(a+b)\mL a$ are defined, i.e. $a + b\geq a,b$.

Since $a\mR b \geq a\mR b$, iff $a\geq (a\mR b) + b$, iff $a\mL(a\mR b)\geq b$.  Similarly it is proved that $a\geq b + (a\mL b)$ and $a\mR (a\mL b)\geq b$.

Let $a + b  = a$, then $0 = (a + b)\mL a \geq b$. Since $0$ is the bottom element, $b$ is equal to $0$.

Since $(a + b) \mR b \geq a$, then $0 \mR b$ is defined and thus $0\geq b$, $b = 0$. Then $0\geq a$, thus $a = 0$.

Let $a\mL b = a$, then $b + (a\mL b)$ is defined and $a\geq b + (a\mL b) = b + a$. On the other hand if $b + a$ is defined then $b + a\geq a$. Thus we get that $b + a = a$ which is possible only if $b = 0$. The converse implication follows from (iii). The proof for right minus is analogous.

Let us assume that $a+b$ and $(a + b) + c$ exist. Then from (iv) of Proposition \ref{prop:gprepea} applied twice we get that $[(a + b) + c]\mR c\geq (a + b)$ and $[[(a+b)+c]\mR c \mR b\geq a$, i.e. $[(a+b)+c]\mR(b+c)\geq a$ and $[(a+b) +c \geq a + (b+c)$ which shows that both $b+c$ and $a + (b + c)$ are defined.

In a similar way, if we assume that $b+c$ and $a + (b + c)$ are defined then $[a + (b+c)]\mL a\geq b+c$ and $[[a + (b+c)]\mL a]\mL b = [a+(b+c)]\mL(a+b)\geq c$, i.e. $a + (b+c)\geq (a+b) + c$ which shows that both $a+b$ and $(a+b) +c$ are defined. The equality follows from the previously proved inequalities.

Let $a\geq b \geq c$. Then $a\geq b\geq (b\mR c) + c$ and $a\mR c\geq b\mR c$. Similarly, $a\geq b\geq c + (b\mL c)$ and $a\mL c\geq b\mL c$.

Let $a\geq b$ and $a + c$ be defined. Then $(a+ c)\mR c\geq a\geq b$, iff $a + c\geq b+c$. Moreover, $(a  + c)\mR (b+c) = ([a + c)\mR c]\mR b\geq a\mR b$.

Let us assume $a\geq b\geq c$. Then $b\geq (b\mR c) + c$, $a\geq a\mR b + b\geq a\mR b + b\mR c + c$ and $a\mR c\geq a\mR b  + b\mR c$.

Let us assume that $(b\mR a)\mL c$ is defined, then $(c + [(b\mR a)\mL c]) + a\leq b$. Since plus is partially associative we get $c + ([b\mR a)\mL c] + a) \leq b$ and $(b\mR a)\mL c\leq (b\mL c)\mR a$.

Similarly, from the existence of $(b\mL c)\mR a$ we are able to infer that $(b\mL c)\mR a\leq (b\mR a)\mL c$. These two inequalities give desired equality.

\end{proof}

\begin{theorem}\label{thm:opconstr}
Let $A$ be a weak generalized pre pseudo effect algebra. Then
\begin{enumerate}[(i)]
\item operation $+$ and partial order $\leq$ uniquely determine operations $\mL$ and $\mR$;
\item either of the operations $\mL$ and $\mR$ uniquely determine the remaining other operations and partial order.
\end{enumerate}
\end{theorem}

\begin{proof}
(i) Let $A$ be a weak generalized pre pseudo effect algebra. And let us fix the partial order and operation $+$. Then knowing only these two informations we are able to reconstruct both kinds of minus operations. First observe that from (GPPEA3) follows that if $a\mR b$ is defined then it is the largest element $z$ such that $a\geq z+b$, i.e. it is the supremum of the set $R_{a,b}:=\{z\in A: z + b \mbox{ is defined and } a\geq z + b \}$. Let $a\mR b$ is defined. Then $a\mR b\geq a\mR b$ and from (GPPEA3) it follows that $(a\mR b) + b$ is defined and $a\geq (a\mR b) + b$. Thus $a\mR b$ belongs to $R_{a,b}$. Let $t\in R_{a,b}$, then $a\geq t + b$ and $a\mR b\geq t$, thus $a\mR b$ is supremum of $R_{a,b}$. Similar reasoning using (GPPEA4) leads to the fact that $a\mL b$ if exists is the supremum of the set $L_{a,b}:=\{z\in A: b + z \mbox{ is defined and } a\geq b + z \}$.

(ii) Let $\mR$ be fixed. Then due to (GPPEA2) we have determined the partial order $\leq$. Now we prove that $+$ can be uniquely defined and then from part (i) the remaining $\mL$ can be defined. Let us consider sets $P_{a,b}:= \{z\in A: a\leq z\mR b\}$ for any pair of $a,b\in A$. Then there are two cases. (A) $P_{a,b} = \emptyset$, (B) $P_{a,b}\neq \emptyset$. If $a + b$ is defined then by Proposition 3.2 (iv) $P_{a,b}$ cannot be empty. Thus the case (A) represents the situation that $a + b$ is undefined. Moreover, $a + b\in P_{a,b}$. Let $t\in P_{a,b}$, then $a\leq t\mR b$ implies by (GPPEA3) $a + b\leq t$. Thus $a + b$ is the infimum of $P_{a,b}$.

In similar fashion starting from $\mL$ by defining $P^l_{a,b}:=\{z\in A: b\leq z\mL a\}$ we are able to reconstruct $+$. 
\end{proof}

We note that when generating finite models, not all left (or right) minuses satisfying axiom ... lead to a weak generalized pre pseudo effect algebra. 

\begin{example}
Let us consider $\{0, a, b, 1\}$ with partial order such that $0$ is bottom, $1$ is top element and $a$ and $b$ are incomparable. Left minus candidate

\begin{center}
\begin{tabular}{|c|c|c|c|c|}
\hline
$\mL$ & \textbf{0} & \textbf{a} & \textbf{b} & \textbf{1} \\ 
\hline
\textbf{0} &   0 & - & - & - \\
\hline
\textbf{a} &   a & 0 & - & - \\
\hline
\textbf{b} &   b & - & 0 & - \\
\hline
\textbf{1} &   1 & a & a & 0 \\
\hline
\end{tabular}
\end{center}
generate according to the algorithm in the previous theorem the following operation $+$:
\begin{center}
\begin{tabular}{|c|c|c|c|c|}
\hline
$+$ & \textbf{0} & \textbf{a} & \textbf{b} & \textbf{1} \\ 
\hline
\textbf{0} &   0 & a & b & 1 \\
\hline
\textbf{a} &   a & 1 & - & - \\
\hline
\textbf{b} &   b & 1 & - & - \\
\hline
\textbf{1} &   1 & - & - & - \\
\hline
\end{tabular}\ .
\end{center}

Indeed, $P^l_{0,0} = \{0,a,b,1\}$, $P^l_{a,0} = P^l_{0,a} = \{a,1\}$, $P^l_{0,b} = P^l_{b,0} = \{b, 1\}$, $P^l_{0,1} = P^l_{1,0} = \{1\}$. Other combinations lead to empty set.

But then $1\mR a$ cannot be defined. Actually, $R_{1,a} = \{a, b\}$ and this set does not have supremum.

\end{example}

\begin{example}
Even $+$ from $\mL$ cannot be sometimes defined.

Let us consider $\mL$ candidate given by the table with depicted considered partial order:
\begin{center}
\begin{tabular}{|c|c|c|c|c|c|}
\hline
$\mL$ & \textbf{0} & \textbf{a} & \textbf{b} & \textbf{c} & \textbf{d} \\ 
\hline
\textbf{0} &   0 & - & - & - & - \\
\hline
\textbf{a} &   a & 0 & - & - & - \\
\hline
\textbf{b} &   b & - & 0 & - & - \\
\hline
\textbf{c} &   c & 0 & a & 0 & - \\
\hline
\textbf{d} &   d & 0 & a & - & 0 \\
\hline
\end{tabular}
\qquad
\raise1.3cm\vbox to 0pt{\hbox{\includegraphics[height=2.5cm]{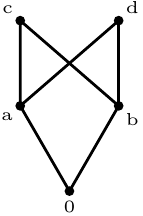}}}
\end{center}

If we try to define $b + a$, then $P^l_{b,a} = \{c, d\}$, but such set does not have infimum.

\end{example}

\begin{example}
Finally, even when $\mL$ successfuly generates $+$ and $+$ generates $\mR$, the resulting structure need not be a weak generalized pre pseudo effect algebra.

Left $\mL$ candidate is given by the table. The other tables contain computed $+$ and $\mR$, respectively.
Partial order is as follows $0 < a < b, c$, where $b$ and $c$ are incomparable.

\begin{center}
\begin{tabular}{|c|c|c|c|c|}
\hline
$\mL$ & \textbf{0} & \textbf{a} & \textbf{b} & \textbf{c} \\ 
\hline
\textbf{0} &   0 & - & - & - \\
\hline
\textbf{a} &   a & 0 & - & - \\
\hline
\textbf{b} &   b & a & 0 & - \\
\hline
\textbf{c} &   c & a & - & 0 \\
\hline
\end{tabular}\quad
\begin{tabular}{|c|c|c|c|c|}
\hline
$+$ & \textbf{0} & \textbf{a} & \textbf{b} & \textbf{c} \\ 
\hline
\textbf{0} &   0 & a & b & c \\
\hline
\textbf{a} &   a & a & - & - \\
\hline
\textbf{b} &   b & - & - & - \\
\hline
\textbf{c} &   c & - & - & - \\
\hline
\end{tabular}\quad
\begin{tabular}{|c|c|c|c|c|}
\hline
$\mR$ & \textbf{0} & \textbf{a} & \textbf{b} & \textbf{c} \\ 
\hline
\textbf{0} &   0 & - & - & - \\
\hline
\textbf{a} &   a & 0 & - & - \\
\hline
\textbf{b} &   b & a & 0 & - \\
\hline
\textbf{c} &   c & a & - & 0 \\
\hline
\end{tabular}
\end{center}

If it was weak generalized pre pseudo effect algebra, then since $a\oplus a = a$, then from  is defined, according Proposition 4.2 (vii), $a$ should equal to $0$, which is a contradiction. Another argument that it is not WGPPEA, (GPPEA3) axiom is violated, namely $a\oplus a = a\leq a$, but $a\mR a = 0\not \geq a$.
\end{example}

Each effect and pseudo-effect algebra can be made into generalized (pseudo-)effect algebra. The same holds for (weak) pre pseudo effect algebras.

The necessary and sufficient conditions for possibility do define operation $\mL$, $\mR$ are summarized in the next theorem. 

\begin{theorem}
Let $(A; +, {}^L, {}^R, 0, 1)$ be a (weak) pre pseudo effect algebra with the partial order $\leq$ defined by (WPPEA3). For any $a, b\in A$, $a\geq b$ let denote the sets $L_{a,b}:= \{k\in A: b + k\leq a\}$ and $R_{a,b}:=\{k\in A: k + b\leq a\}$ and let assume that such sets posess supremum denoted $l_{a,b}$, $r_{a,b}$, respectively. Then if we define $a\mL b = l_{a,b}$, $a\mR b = r_{a,b}$ for $a\geq b$ and otherwise undefined, the structure $(A; +, \mL, \mR, 0)$ is a (weak) generalized pre pseudo effect algebra.
\end{theorem}

\begin{proof}
Since $0\in L_{a,b}, R_{a,b}$, for $a\geq b$, the conserned sets are non-empty so the assumption defined actual values for $l_{a,b}$ and $r_{a,b}$, $a\geq b$.

We directly prove each of the axiom of generalized pre pseudo effect algebras.

(GPPEA1) According to Proposition 2.4 (ii) and (iii) we have $\{k\in A: a + k\leq a\} = \{k\in A: a + k = a\} = \{0\}$. Thus $l_{a,a} =  0$ for any $a\in A$. Using Proposition 2.4 (ii) and (iv) we prove that $r_{a,a} = 0$.

(GPPEA2) According to definition $\mR$ and $\mL$ are defined only for pairs $a\geq b$.

(GPPEA3) Let $a\mR b$ be defined and $a\mR b \geq c$. Then $r_{a,b}\geq c$. Since $r_{a,b}$ is the element of $R_{a,b}$ then $r_{a,b} + b$ is defined. Since $c\leq r_{a,b}$ by Proposition 2.4 (viii) we have $c + b\leq r_{a,b} + b\leq a$. Conversely, let $c + b$ be defined and $a\geq c + b$. Then $c\in R_{a,b}$ and $a\mR b = r_{a,b}\geq c$.

We need to prove equality of the following two sets

$$R_{r_{a,b},c} = \{k: k + c \leq r_{a,b}\},$$
$$R_{a,c + b} = \{k: k + (c + b) \leq a\}.$$

Let $t\in R_{r_{a,b},c}$. Then $t + c \leq r_{a,b}$. Since $r_{a,b} + b$ is defined $(t + c) + b = t + (c + b) \leq r_{a,b} + b\leq a$, thus $t\in R_{a, c+b}$.

Let now $t\in R_{a, b+c}$. Then $t + (c + b)  = (t + c) + b \leq a$. Thus $t + c\in R_{a,b}$ and $t + c \leq r_{a,b}$. This implies $t\in R_{r_{a,b},c}$. 

So we have proved that $(a\mR b)\mR c = a\mR (c + b)$.

Axiom (GPPEA4) can be proved by similar reasoning using sets $L_{a,b}$.

Since (PEA) condition concerns only operation $+$, it is trivially preserved.

Necessity proof, let us assume that (W)PPEA can be made into (W)GPPEA. Then preserved order and operation $+$ allow us to compute according to Theorem \ref{thm:opconstr} the remaing $\mL$ and $\mR$. The construction conctains precisely sets $L_{a,b}$ and $R_{a,b}$ which are either empty or posses supremum. As it was shown it the beginning of the proof, for $a\geq b$, such sets are non-empty and thus they need to posess supremum. For other pairs of $a$, $b$, the sets are empty.
\end{proof}

It is interesting to note, that similarly as in the cases of effect and pseudo-effect algebras operations $\mL$ and $\mR$ are defined explicitely. Namely, for $a\geq b$, $a\mL b = (a^L + b)^R$ and $a\mR b = (b + a^R)^L$. Since the order, both expressions are well defined. For other pairs of $a$, $b$ the operations are undefined. 

According to previous theorem, it is sufficient to prove that such elements are for $a\geq b$, suprema of $L_{a,b}$ and $R_{a,b}$, respectively. Since for each $k\in L_{a,b}$ $b + k \leq a$, then from (WPPEA3) $a^L + (b + k)$ is defined, which means that $(a^L + b) + k = (a^L + b)^{RL} + k$ is defined and again from (WPPEA3) this means that $k\leq (a^L + b)^R$. Thus any element of $L_{a,b}$ is less than or equal to $(a^L + b)^R$. It suffice to show that the expression belongs to $L_{a,b}$. The reasoning is in backward direction, $b\leq a^L + b$, i.e. $b + (a^L + b)^R$ is defined, and from (WPPEA2) $(a^L + b) + (a^L + b)^R = a^L + (b + (a^L + b)^R)$ is defined, i.e. $b + (a^L + b)^R \leq a$. For $R_{a,b}$ it is shown similarly.

\begin{proposition}\label{prop:unitization}
Let $(A; +, \mR, \mL, 0)$ be a generalized pre pseudo effect algebra. Let us consider disjunctive copy of $A$, denoted as $A^*$, and let us denote its elements as $a^*$ for each corresponding $a\in A$. Let us define operation $+_p$ as following:
$a +_p b$ is defined, iff $a + b$ is defined and $a +_p b = a + b$, $a, b\in A$; $a +_p b^*$ is defined, iff $b\geq a$ and $a +_p b^* = (b\mR a)^*$; $b^* +_p a$ is defined, iff $b\geq a$ and $b^* +_p a = (b\mL a)^*$; $a^* +_p b^*$ is never defined. For each element $a\in A$, let $a^R = a^L = a^*$ and for each element $a^*\in A^*$ $(a^*)^R = (a^*)^L = a$. Then $(A\cup A^*; +_p, {}^R, {}^L, 0, 0^*)$ is a weak pre pseudo effect algebra.
\end{proposition}

First it is necessary to verify that $(A\cup A^*; +_p, 0)$ is partial monoid with neutral element $0$. Since $a + 0 = 0 + a = a \mR 0 = a \mL 0 = a$ it is easy to see that $0$ is neutral element with respect to $+_p$.

For checking associativity, we have 4 possible cases. If all elements are from $A$ it follows from partial associativity of $+$. $a^*\in A^*$, $b, c\in A$ it follows from the fact that $(c\mR b)\mR a = c\mR (a + b)$. The case $c^*\in A^*$, $a, b\in A$ follows from the fact that $a\mL (b + c) = (a\mL b)\mL c$. And finally the case $a, c\in A$, $b^*\in A^*$ follows from the fact that $(b\mR a)\mL c = (b\mL c)\mR a$.

$a +_p a^R = a +_p a^* = (a\mR a)^* = 0^* = (a\mL a)^* = a^* +_p a = a^L +_p a$.

Since $b\mR a$ is defined is equivalent to $a +_p b^R$ is defined and $b\mL a$ is defined is equivalent to $b^L +_p a$ is defined. The resulting relation is equal to $\leq$ from $A$ which is partial order.

Let us assume the $0^* +_p a$ or $a +_p 0^*$  is defined then $0\geq a$, i.e. $a = 0$.

\begin{example}
Let us consider a partially ordered set $(A, \leq)$. When we define operation $x\mL y = x\mR y = 0$, for $y\leq x$ and undefined otherwise and operation $x + y$ is defined, if either $x$ or $y$ are equal to $0$ and in that case the result is the other element. Then $(A; +, \mL, \mR, 0)$ is a generalized pre pseudo effect algebra. Since the operations $\mL$ and $\mR$ coincide such algebras are also generalized pre effect algebras.
\end{example}

\begin{example}
The smallest strict generalized weak pre pseudo effect algebra, i.e. it is not weak pre pseudo effect algebra, is the following. Let us consider 4 element set $\{0, a, b, 1\}$. whose partial order is isomorphic to the order of two-element Boolean algebra, i.e. $0$ is a bottom element, $1$ is a top element and $a$ and $b$ are mutually incomparable. Let us consider operation $+$ defined for pairs $(0,x)$, $(x,0)$ with the result of $x$ and the only other pair which is summable is $b + a$ and it equals to $1$. Corresponding left minus $\mL$ and right minus $\mR$ are trivially defined for pairs $(x,0)$, $(x,x)$. Moreover, $1\mL a = 0$, $1\mL b =  a$ and $1\mR a = b$, $1\mL b = 0$. It is not weak pre pseudo effect algebra since for the element $a$ there is no element $x$ (a candidate for right orthosupplement) such that $a + x = 1$.
\end{example}

\section{Congruences}

Let us consider generalized pre pseudo effect algebras and $\sim$ be a relation of equivalence on underlying set. For any element $a\in A$ we denote $[a]_{\sim}$ the set of all elements which are in relation $\sim$ with $a$, i.e. $[a]_{\sim} = \{t\in A: t\sim a\}$. If the relation will be known, then we will shortly write $[a]$ for the equivalence class of $a$.

We will try to explain the rationale for a condition which need to be added to get some kind of congruence relation.

Let fix the relation of equivalence and let us consider two classes $[a]$ and $[b]$. Operations on sets $[a]$ and $[b]$ are as follows:

\begin{equation}\label{eq:plus}
[a] + [b] = \{m = a' + b'; a'\in [a], b'\in [b]\},\\
\end{equation}
\begin{equation}\label{eq:leftminus}
[a] \mathrel{\mL} [b] = \{m = a' \mL b'; a'\in [a], b'\in [b]\},\\
\end{equation}
\begin{equation}\label{eq:rightminus}
[a] \mathrel{\mR} [b] = \{m = a' \mR b'; a'\in [a], b'\in [b]\},
\end{equation}
provided that the operations in question are defined.

We need to achieve to goals, first, that all elements from the particular set belong to the one class of equivalence and for the second, that each element of this class of equivalence is realized in this way.

To achieve the first goal, we can use the method from effect algebras, i.e. we can define weak congruence conditions for operation $+$, $\mL$ and $\mR$.

$a_1 \sim a_2, b_1\sim b_2$, $a_1\mathbin{\mbox{op}} b_1$ and $a_2 \mathbin{\mbox{op}} b_2$ are defined, then $a_1\mathbin{\mbox{op}} b_1 \sim a_2\mathbin{\mbox{op}} b_2$, where op stands for one of the +, $\mL$, $\mR$.

These conditions allows us to state that elements $[a]\mathbin{\mbox{op}} [b]$ are from one equivalence class (if it is non-empty set), but it is not implied that whole equivalence class $[a]\mathbin{\mbox{op}} [b]$ is realized in this way. In general, it is just a subset.

Let us consider the following conditions: $[a]\mathbin{\mbox{op}} [b]$ is defined (the set is non-empty) and it is a subset of some class of $[t]$, where $t \sim a'\mathbin{\mbox{op}} b'$, for some $a'\sim a$ and $b'\sim b$. Let $t'\in [t]$, then there are $a''\in [a]$ and $b''\in [b]$ such that $t' = a''\mathbin{\mbox{op}} b''$. These conditions lead to the desired fact, that if $[a]\mathbin{\mbox{op}} [b]$ is a non-empty set, then if the op is defined as above it is the whole equivalence class.

To summarize this:

\begin{definition} Let $(A; +, \mL, \mR, 0)$ be a (weak) generalized pre pseudo effect algebra. Let $\sim$ be a relation of equivalence on $A$. Then we say that it is a \emph{weak congruence}, if the following conditions are satisfied (where op is subsequently one of the binary operations $+$, $\mL$, $\mR$):
\begin{itemize}
\item $a_1\sim a_2$, $b_1\sim b_2$ and $\exists$ $a_1\mathbin{\mbox{op}}b_1$ and $\exists$ $a_2\mathbin{\mbox{op}}b_2$, then $a_1\mathbin{\mbox{op}}b_1 \sim a_2\mathbin{\mbox{op}}b_2$.
\end{itemize}

We say that $\sim$ is a \emph{congruence}, if it is a weak congruence and
\begin{itemize}
\item if $[a]\mathbin{\mbox{op}} [b] \subseteq[t]$ is non-empty set, then for each $t'\in [t]$ there are $a'\in [a]$ and $b'\in [b]$ such that $t' = a'\mathbin{\mbox{op}} b'$,
\item if $[a]\mathbin{\mbox{op}} [b]$ is defined, then for any $a'\in [a]$ there is $b'\in [b]$ such that $a'\mathbin{\mbox{op}} b'$ is defined,
\item if $[a]\mathbin{\mbox{op}} [b]$ is defined, then for any $b'\in [b]$ there is $a'\in [a]$ such that $a'\mathbin{\mbox{op}} b'$ is defined,
\end{itemize}
where again op is one of the binary operations.
\end{definition}

In the light of the previous reasoning we say that $[a]\mathbin{\mbox{op}}[b]$ is defined, iff the underlying set defined by equations (\ref{eq:plus})-(\ref{eq:rightminus}) is non-empty. If we consider a congruence relation $\sim$, then we can simply form a factor algebra $A/\sim = \{[a]; a\in A\}$. Then we can try to prove that $(A/\sim; +, \mL, \mR, [0])$ is again (weak) generalized pre effect algebra.

Let us try to prove some preliminary results:

Let $[a] + [b] = [0]$, then $[a] = [b] = [0]$.

Indeed, for any $c\in [0]$ there are $a'\in [a]$ and $b'\in [b]$ such that $c = a' + b'$. Since $0\in [0]$, there are $a''\in [a]$, $b''\in [b]$ such that $a'' + b'' = 0$. From Proposition it follows that $a'' = b'' = 0$ and $0\in [a], [b]$. Thus $[a] = [b] = [0]$.

Let $[a]\mL [b]$ and $[b]\mL [a]$ be defined. Then $[a] \mL [b] = [0]$:

Let us denote $[c] := [a]\mL [b]$ and $[d] := [b]\mL [a]$. For $a\in [a]$ there is $b'\in [b]$ such thant $a\mL b'$ is defined and $a\mL b' \sim c$. Similarly, for $a\in [a]$ there is $b''\in [b]$ such that $b''\mL a$ is defined and $b''\mL a \sim d$. This means that $b''\geq a \geq b'$. From Proposition we have that $b''\mL b'\geq (a\mL b') + (b''\mL a)$. Since $b'', b\in [b]$ $b''\mL b'\sim 0$ and $b''\mL b' \geq c' + d'$ for some $c'\in [c]$ and $d'\in [d]$. We see that $[c] + [d]$ is defined. For any $t\in [0]$ there are $e\in [c]$ and $f\in [d]$ such that $t\geq e + f$. For $t = 0$ we have $e'\in [c]$ and $f'\in [d]$ such that $0 \geq e' +f'$ which means that $e' + f' = 0$ and thus $e' = f' = 0$. Thus $[c] = [d] = [0]$.

\medskip\textbf{Unsorted thoughts:}\medskip

The attempt to prove that factor algebra is again gwprepea failed. It was not possible to proved that if $[a]\mL [b]$ and $[b]\mL[a]$ are defined, then $[a] = [b]$. That is, the antisymmetry in the definition of partial order.

Since $x \mL y = 0$, or ($x \mR y = 0$) does not imply that $x$ and $y$ equal, we have just implication: if $[x] = [y]$, then $[x] \mL [y] = [0]$, or $[x]\mR [y] = [0]$. Taking into consideration the trivial congruence ($x=x$), we have that in general $[x]\mL [y] = [0]$ (or $[x]\mR [y] = [0]$) does not imply $[x] = [y]$, i.e. $x\sim y$.

From the models generated, operation $+$ does not uniquely imply the remaining structure. As in example above, the trivial operation $+$ containing trivial sums with $0$ (at least one of the summand is $0$) can be defined for arbitrary order.

On the other hand, at least for models generated (up to 7 elements and partially for 8), once the $\mL$ (or $\mR$) is fixed there is the unique structure of gwprepea having this operation. It seems that fixing one of the minus operations implies the remaing two operations. Unfortunately, for now, I was unable to prove it from axioms. Other catch might be that this only holds for finite structures, or the counterexample need to be of much larger cardinality.

If $a + b = a + c$ and $b\geq c$, then $b\mL c = 0$.

If $b + a = c + a$ and $b\geq c$, then $b\mR c = 0$.

\section{Modifications of Riesz decomposition property (RDP) and Riesz interpolation property (RIP)}

Let us recall definitions of Riesz decomposition and interpolation properties.

\begin{definition}
Let for any $a_1, a_2, b_1, b_2\in A$ holding $a_1 + a_2 = b_1 + b_2$, there are elements $c_{11}, c_{12}, c_{21}, c_{22}\in A$ such that the sums in rows and columns equal to respective elements:
\begin{center}\begin{tabular}{ccc}
 & $b_1$ & $b_2$ \\
$a_1$ & $c_{11}$ & $c_{12}$ \\
$a_2$ & $c_{21}$ & $c_{22}$ \\
\end{tabular}
\end{center}
That is $a_1 = c_{11} + c_{12}$, $a_2 = c_{21} + c_{22}$, $b_1 = c_{11} + c_{21}$ and $b_2 = c_{12} + c_{22}$. Then $A$ satisfies the so called \emph{Riesz decomposition property}, (RDP) for short.
\end{definition}

\begin{definition}
If for any $a, b_1, b_2\in A$ such that $a \leq b_1 + b_2$, there are elements $a_1, a_2\in A$ satisfying $a_1\leq b_1$, $a_2\leq b_2$ and $a = a_1 + a_2$. Then $A$ satisfies \emph{Riesz interpolation property}, (RIP) for short.
\end{definition}

Let us consider the following examples demonstrating, that (RDP) does not imply (RIP). The second example has linear order.

\begin{example}$\mbox{}$\\
\begin{center}
\begin{tabular}{|c|c|c|c|c|}
\hline
$+$ & \textbf{0} & \textbf{1} & \textbf{2} & \textbf{3}\\
\hline
\textbf{0} & 0 & 1 & 2 & 3\\
\hline
\textbf{1} & 1 & . & . & .\\
\hline
\textbf{2} & 2 & . & 3 & .\\
\hline
\textbf{3} & 3 & . & . & .\\
\hline
\end{tabular}
\quad
\begin{tabular}{|c|c|c|c|c|}
\hline
$\mL = \mR$ & \textbf{0} & \textbf{1} & \textbf{2} & \textbf{3}\\
\hline
\textbf{0} & 0 & . & . & .\\
\hline
\textbf{1} & 1 & 0 & . & .\\
\hline
\textbf{2} & 2 & . & 0 & .\\
\hline
\textbf{3} & 3 & 0 & 2 & 0\\
\hline
\end{tabular}
\end{center}
\end{example}

\begin{example}$\mbox{}$\\

\begin{center}
\begin{tabular}{|c|c|c|c|c|}
\hline
$+$ & \textbf{0} & \textbf{1} & \textbf{2} & \textbf{3}\\
\hline
\textbf{0} & 0 & 1 & 2 & 3\\
\hline
\textbf{1} & 1 & 3 & . & .\\
\hline
\textbf{2} & 2 & . & . & .\\
\hline
\textbf{3} & 3 & . & . & .\\
\hline
\end{tabular}
\quad
\begin{tabular}{|c|c|c|c|c|}
\hline
$\mL = \mR$ & \textbf{0} & \textbf{1} & \textbf{2} & \textbf{3}\\
\hline
\textbf{0} & 0 & . & . & .\\
\hline
\textbf{1} & 1 & 0 & . & .\\
\hline
\textbf{2} & 2 & 0 & 0 & .\\
\hline
\textbf{3} & 3 & 1 & 0 & 0\\
\hline
\end{tabular}
\end{center}

In both cases there are no nontrivial decompositions for $a_1 + a_2 = b_1 + b_2$. Indeed, there is either at least one of the elements equal to 0, or the decomposition lead to the case $a_1 = b_1$, $a_2 = b_2$. To see that RIP property fails it is sufficient to consider $1\leq 2 + 2 = 3$, but there are no elements satisfing $1 = a_1 + a_2$. (In the second case $2\leq 1 + 1 = 3$.)
\end{example}

On the other hand, there is also the example of RIP, which does not satisfy RDP:

\begin{center}
\begin{tabular}{|c|c|c|c|c|c|}
\hline
$+$ & \textbf{0} & \textbf{1} & \textbf{2} & \textbf{3} & \textbf{4}\\
\hline
\textbf{0} & 0 & 1 & 2 & 3 & 4\\
\hline
\textbf{1} & 3 & 4 & . & . & .\\
\hline
\textbf{2} & 4 & 4 & . & . & .\\
\hline
\textbf{3} & . & . & . & . & .\\
\hline
\textbf{4} & . & . & . & . & .\\
\hline
\end{tabular}
\quad
\begin{tabular}{|c|c|c|c|c|c|}
\hline
$\mL = \mR$ & \textbf{0} & \textbf{1} & \textbf{2} & \textbf{3} & \textbf{4}\\
\hline
\textbf{0} & 0 & . & . & . & .\\
\hline
\textbf{1} & 1 & 0 & . & . & .\\
\hline
\textbf{2} & 2 & 0 & 0 & . & .\\
\hline
\textbf{3} & 3 & 1 & . & 0 & .\\
\hline
\textbf{4} & 4 & 2 & 2 & 0 & 0\\
\hline
\end{tabular}
\end{center}

Indeed, for $1 + 2 = 2 + 1$ there is no RDP decomposition. To see that RIP is satisfied, it is sufficient to check it for the case $3\leq2+2 = 4$. Since $1\leq 2$, the desired RIP decomposition is $3 = 1 + 1$. The other cases are trivial. (For the completeness, we will list the trivial cases. 1) One of the $b_i$ is equal to $0$; 2) $a$ is less than or equal to one of the $b_i$; 3) $a$ is equal to $b_1 + b_2$.)

Since $+$ no longer defines partial order on $A$ the (RDP) and (RIP) properties no longer depend only on partial order. Indeed, there are generalized weak pre pseudo effect algebras, having the same order, but one satisfies the (RDP) (or (RIP)), and the other does not.

It may be interesting to try to investigate alternative definition of (RDP) (or (RIP)) properties. The main idea is to replace equality of two elements  with the property that their difference (left or right) is equal to $0$. If $a = b$, then certainly $a \mL b = 0$ (and $a\mR b = 0$), but from $a\mL b = 0$, or $a\mR b = 0$ in general does not follow $a = b$.

Left modified RIP, (LmodRIP): $ (a \mL a_1) \mL a_2 = 0 = a \mL (a_1 + a_2)$

Right modified RIP, (RmodRIP): $ (a \mR a_2) \mR a_1 = 0 = a \mR (a_1 + a_2)$

Left-Right modified RIP, (LRmodRIP) $ (a \mL a_1)\mR a_2 = 0$.

According to Proposition 3.2 (xiv), the concept of Right-Left modified RIP, i.e. $ (a \mR a_2) \mL a_1 = 0$ coincides with Left-Right modified RIP.

The following examples present the LmodRIP which is neither RmodRIP nor LRmodRIP. The "transposed" example, Example \ref{ex:trans}, presents RmodRIP which is not LmodRIP or LRmodRIP.

\begin{example}$\mbox{}$\\

\begin{center}
\begin{tabular}{|c|c|c|c|c|c|c|c|}
\hline
$+$ & \textbf{0} & \textbf{1} & \textbf{2} & \textbf{3} & \textbf{4} & \textbf{5} & \textbf{6} \\
\hline
\textbf{0} & 0 & 1 & 2 & 3 & 4 & 5 & 6 \\
\hline
\textbf{1} & 1 & . & 5 & 6 & 6 & . & . \\
\hline
\textbf{2} & 2 & . & 4 & 4 & . & . & . \\
\hline
\textbf{3} & 3 & . & . & . & . & . & . \\
\hline
\textbf{4} & 4 & . & . & . & . & . & . \\
\hline
\textbf{5} & 5 & . & 6 & 6 & . & . & . \\
\hline
\textbf{6} & 6 & . & . & . & . & . & . \\
\hline
\end{tabular}
\quad
\begin{tabular}{|c|c|c|c|c|c|c|c|}
\hline
$\mL$ & \textbf{0} & \textbf{1} & \textbf{2} & \textbf{3} & \textbf{4} & \textbf{5} & \textbf{6} \\
\hline
\textbf{0} & 0 & . & . & . & . & . & . \\
\hline
\textbf{1} & 1 & 0 & . & . & . & . & . \\
\hline
\textbf{2} & 2 & . & 0 & . & . & . & . \\
\hline
\textbf{3} & 3 & . & 0 & 0 & . & . & . \\
\hline
\textbf{4} & 4 & . & 3 & 0 & 0 & . & . \\
\hline
\textbf{5} & 5 & 2 & 0 & . & . & 0 & . \\
\hline
\textbf{6} & 6 & 4 & 3 & 0 & 0 & 3 & 0 \\
\hline
\end{tabular}
\end{center}

\begin{center}
\vbox to 0cm {\vspace{-1.5cm}\hbox{\includegraphics{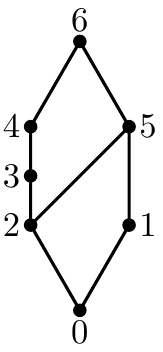}}}
\begin{tabular}{|c|c|c|c|c|c|c|c|}
\hline
$\mR$ & \textbf{0} & \textbf{1} & \textbf{2} & \textbf{3} & \textbf{4} & \textbf{5} & \textbf{6} \\
\hline
\textbf{0} & 0 & . & . & . & . & . & . \\
\hline
\textbf{1} & 1 & 0 & . & . & . & . & . \\
\hline
\textbf{2} & 2 & . & 0 & . & . & . & . \\
\hline
\textbf{3} & 3 & . & 0 & 0 & . & . & . \\
\hline
\textbf{4} & 4 & . & 2 & 2 & 0 & . & . \\
\hline
\textbf{5} & 5 & 0 & 1 & . & . & 0 & . \\
\hline
\textbf{6} & 6 & 0 & 5 & 5 & 1 & 0 & 0 \\
\hline
\end{tabular}
\end{center}

Indeed, $6\leq 6 = 1 + 3$, but for any $s\leq 1$, $t\leq 3$, $(6\mL s)\mR t \neq 0$, thus LRmodRIP fails.

%% We have $s \in \{1, 0\}$, $t\in \{3, 2, 0\}$:
%% 
%% $(6\mL 1)\mR 3 = 4\mR 3 = 2$,\qquad
%% $(6\mL 1)\mR 2 = 4\mR 2 = 2$,\qquad
%% $(6\mL 1)\mR 0 = 4\mR 0 = 4$,
%% 
%% $(6\mL 0)\mR 3 = 6\mR 3 = 5$,\qquad
%% $(6\mL 0)\mR 2 = 6\mR 2 = 5$,\qquad
%% $(6\mL 0)\mR 0 = 6\mR 0 = 6$.

Also RmodRIP fails, $4\leq 1 + 3 = 6$, but there is no $t\leq 1$, $s\leq 3$, such that $t + s = 4=: e$ (the only element such that $4\mR e = 0$).

\end{example}

\begin{example}$\label{ex:trans}\mbox{}$\\

\begin{center}
\begin{tabular}{|c|c|c|c|c|c|c|c|}
\hline
$+$ & \textbf{0} & \textbf{1} & \textbf{2} & \textbf{3} & \textbf{4} & \textbf{5} & \textbf{6} \\
\hline
\textbf{0} & 0 & 1 & 2 & 3 & 4 & 5 & 6 \\
\hline
\textbf{1} & 1 & . & . & . & . & . & . \\
\hline
\textbf{2} & 2 & 5 & 4 & . & . & 6 & . \\
\hline
\textbf{3} & 3 & 6 & 4 & . & . & 6 & . \\
\hline
\textbf{4} & 4 & 6 & . & . & . & . & . \\
\hline
\textbf{5} & 5 & . & . & . & . & . & . \\
\hline
\textbf{6} & 6 & . & . & . & . & . & . \\
\hline
\end{tabular}
\quad
\begin{tabular}{|c|c|c|c|c|c|c|c|}
\hline
$\mL$ & \textbf{0} & \textbf{1} & \textbf{2} & \textbf{3} & \textbf{4} & \textbf{5} & \textbf{6} \\
\hline
\textbf{0} & 0 & . & . & . & . & . & . \\
\hline
\textbf{1} & 1 & 0 & . & . & . & . & . \\
\hline
\textbf{2} & 2 & . & 0 & . & . & . & . \\
\hline
\textbf{3} & 3 & . & 0 & 0 & . & . & . \\
\hline
\textbf{4} & 4 & . & 2 & 2 & 0 & . & . \\
\hline
\textbf{5} & 5 & 0 & 1 & . & . & 0 & . \\
\hline
\textbf{6} & 6 & 0 & 5 & 5 & 1 & 0 & 0 \\
\hline
\end{tabular}
\end{center}

\begin{center}
\begin{tabular}{|c|c|c|c|c|c|c|c|}
\hline
$\mR$ & \textbf{0} & \textbf{1} & \textbf{2} & \textbf{3} & \textbf{4} & \textbf{5} & \textbf{6} \\
\hline
\textbf{0} & 0 & . & . & . & . & . & . \\
\hline
\textbf{1} & 1 & 0 & . & . & . & . & . \\
\hline
\textbf{2} & 2 & . & 0 & . & . & . & . \\
\hline
\textbf{3} & 3 & . & 0 & 0 & . & . & . \\
\hline
\textbf{4} & 4 & . & 3 & 0 & 0 & . & . \\
\hline
\textbf{5} & 5 & 2 & 0 & . & . & 0 & . \\
\hline
\textbf{6} & 6 & 4 & 3 & 0 & 0 & 3 & 0 \\
\hline
\end{tabular}
\end{center}

Mirror arguments of the previous example. Since in Example~\ref{ex:trans} are operations $\mL$ and $\mR$ exchanged.

Indeed, $6\leq 6 = 3 + 1$, but for any $s\leq 3$, $t\leq 1$, $(6\mL s)\mR t \neq 0$, thus LRmodRIP fails.

Also LmodRIP fails, $4\leq 3 + 1 = 6$, but there is no $t\leq 3$, $s\leq 1$, such that $t + s = 4=: e$ (the only element such that $4\mL e = 0$).

\end{example}

\medskip

Situation for RDP is little bit more complicated. Basically, we have 5 equalities to satisfy and each equality can be modified as in the case of modified RIP to 3 kinds. Thus counting unmodified + 3 other cases we have $4^5  - 1= 1023$ possibilities to modify RDP property (1 case leads exactly to RDP).


\begin{thebibliography}{XX}\parskip=-1mm
\bibitem{DvVe} Dvure\v{c}enskij, A.---Vetterlein, T.: \textit{Pseudoeffect algebras. I. Basic properties.} Int. J. Theor. Phys.  \textbf{40} (2001) 83--99.
\bibitem{FoBe} Foulis, D.---Bennett, M. K.: \textit{Effect algebras and unsharp equantum logics}, Found. Phys. \textbf{24} (1994), 1331--1352.
\bibitem{HePu} Hedl\'{\i}kov\'{a}, J.---Pulmnannov\'{a}, S.: \textit{Generalized difference posets and ortholattices}, Acta Math. Univ. Comenianae \textbf{45} (1996), 247--279.
\bibitem{ChaKu} Chajda, I.---K\"uhr, J.: \textit{A generalization of effect algebras and ortholattices}, Math. Slovaca \textbf{62}, no. 6, (2012), 1045--1062. doi: \texttt{10.2478/s12175-012-0063-4}.
\bibitem{KoCh} K\^opka, F---Chovanec, F.: \textit{D-posets}, Math. Slovaca \textbf{44} (1994), 21--34.
\bibitem{PuVi} Pulmannov\'{a}, S.---Vincekov\'{a}, E.: \textit{Riesz ideals in generalized effect algebras and in their unitizations.}, Algebra Universalis \textbf{57} (2007), 393--417.
\end{thebibliography}
\end{document}